\documentclass{amsart}
\usepackage{amsmath}
\usepackage{amsthm}
\usepackage{amssymb}
\usepackage{amsfonts}
\usepackage{hyperref}
\usepackage{paralist}
\usepackage{amsfonts}
\usepackage{mathrsfs}
\usepackage{graphicx}
\usepackage{url}
\usepackage{comment}
\usepackage{xspace}

\numberwithin{equation}{section}

\newtheorem{PROP}{Proposition}[section]
\newtheorem{THM}[PROP]{Theorem}
\newtheorem{LM}[PROP]{Lemma}
\newtheorem{COR}[PROP]{Corollary}
\theoremstyle{definition}
\newtheorem{DEF}[PROP]{Definition}

\newtheorem{REM}[PROP]{Remark}

\newcommand{\powerset}[1]{\mc P(#1)}
\newcommand{\Zobr}[3]{\ensuremath{#1\colon #2\to #3}}

\newcommand{\RR}{\ensuremath{\mathbb R}}
\newcommand{\NN}{\ensuremath{\mathbb N}}
\newcommand{\vp}{\ensuremath{\varphi}}
\newcommand{\ve}{\ensuremath{\varepsilon}}
\newcommand{\sm}{\ensuremath{\setminus}}
\newcommand{\emps}{\ensuremath{\emptyset}}

\newcommand{\limti}[1]{\lim\limits_{#1\to\infty}}

\newcommand{\dcc}[1]{\ensuremath{\lfloor #1 \rfloor}} 
\newcommand{\hcc}[1]{\ensuremath{\lceil #1 \rceil}} 
\newcommand{\limsupti}[1]{\limsup\limits_{#1\to\infty}}

\newcommand{\Ra}{\ensuremath{\Rightarrow}}

\newcommand{\abs}[1]{\lvert#1\rvert}
\newcommand{\absl}[1]{\left\lvert#1\right\rvert}
\newcommand{\norm}[1]{\lVert#1\rVert}
\newcommand{\ol}[1]{\ensuremath{\overline{#1}}}
\newcommand{\ul}[1]{\ensuremath{\underline{#1}}}
\newcommand{\inv}[1]{#1^{-1}}

\newcommand{\intrv}[2]{\ensuremath{[#1,#2]}}
\newcommand{\intrvr}[2]{(#1,#2]}
\newcommand{\intrvl}[2]{\ensuremath{[ #1,#2 )}}
\newcommand{\seqq}[3]{({#1}_{#2})_{#3=1}^\infty}
\newcommand{\seq}[2]{\seqq{#1}{#2}{#2}}

\newcommand{\mc}[1]{\ensuremath{\mathcal{#1}}}

\newcommand{\FF}{\mathcal{F}}

\newcommand{\GG}{\mathcal{G}}

\newcommand{\Pdots}[3]{\ensuremath{#1_{#2}+\dots+#1_{#3}}}

\newcommand{\Ldots}[3]{\ensuremath{#1_{#2},\ldots,#1_{#3}}}

\newcommand{\Ces}{\mathcal{C}}
\newcommand{\Cesex}{\widehat{\mathcal{C}}}
\newcommand{\Te}{T}

\newcommand{\Su}{\Sigma}
\newcommand{\Sui}[2]{\Su_{\intrvr{#1#2}{#2}}}

\newcommand{\thenul}{\theta_0}
\newcommand{\theo}{\theta'}
\newcommand{\thed}{\theta''}

\DeclareMathOperator*{\Flim}{\mathcal{F}-lim}

\DeclareMathOperator*{\Glim}{\mathcal{G}-lim}

\DeclareMathOperator*{\Gklim}{\mathcal{G}_{\mathit{k}}-lim}

\newcommand{\ccohull}{\operatorname{\overline{co}}}

\newcommand{\cco}{\operatorname{\overline{co}}}

\newcommand{\uud}{\ensuremath{\overline{\overline{d}}}}
\newcommand{\ld}{\ensuremath{\underline{d}}}
\newcommand{\lld}{\ensuremath{\underline{\underline{d}}}}
\newcommand{\uda}[1]{\ensuremath{\overline{d_{#1}}}}
\newcommand{\lda}[1]{\ensuremath{\underline{d_{#1}}}}
\newcommand{\udi}{\uda{\infty}}
\newcommand{\ldi}{\lda{\infty}}
\newcommand{\du}{\ensuremath{d^*}}
\newcommand{\dl}{\ensuremath{d_*}}
\newcommand{\udp}{\ensuremath{\overline{d}_P}}
\newcommand{\ldp}{\ensuremath{\underline{d}_P}}
\newcommand{\DD}{\mc{D}}
\newcommand{\DM}{\widehat{\mc{D}}}

\newcommand{\enu}{\renewcommand{\theenumi}{\roman{enumi}}\renewcommand{\labelenumi}{{\rm (\theenumi)}}}

\begin{document}

\author{Peter Letavaj}
\thanks{The first author was supported by Young Researchers' Support Program of Slovak University of Technology in Bratislava.}
\author{Ladislav Mi\v{s}\'\i{k}}
\thanks{The second author was supported by the project P201/12/2351 of GA\v CR.}
\author{Martin Sleziak}
\thanks{The third author was supported by grant VEGA 1/0608/13.}

\email{\tt peterletavaj@googlemail.com, ladislav.misik@osu.cz, sleziak@fmph.uniba.sk}

\title{Extreme points of the set of density measures}

\date{\today}
\maketitle

\begin{abstract}
We study finitely additive measures on the set $\mathbb N$ which extend the asymptotic density (density measures). We show that there is a one-to-one correspondence between density measures and positive functionals in $\ell_\infty^*$, which extend Ces\`{a}ro mean. Then we study maximal possible value attained by a density measure for a given set $A$ and the corresponding question for the positive functionals extending Ces\`{a}ro mean.
Using the obtained results, we can find a set of functionals such that their closed convex hull in $\ell_\infty^*$ with weak${}^*$ topology is precisely the set of all positive functionals extending Ces\`{a}ro mean.
This also describes a set of density measures, from which all density measures can be obtained as the closed convex hull.
\end{abstract}

\section{Introduction}

Asymptotic density is a very natural way to measure size of subsets of $\NN$. One of the drawbacks of this concept is that there are sets that do not have asymptotic density. This problem leads us to studying finitely additive measures which extend the asymptotic density and are defined for all subsets of $\NN$. Such measures are called \emph{density measures} and they were studied by several authors, for example \cite{BFPR,MAHARAM,SALTIJD,VANDOUWEN}. Density measures (and other types of densities on $\NN$) have applications, for example, in theory of social choice \cite{FEYMAY,LAUWERS,SUREKHABHARAO}.

If $\mu$ is a density measure and a set $A\subseteq\NN$ has asymptotic density, then clearly $\mu(A)=d(A)$. But for sets not having asymptotic density, it might be interesting to find the maximal and minimal possible values of $\mu(A)$. This problem was posed in \cite{FEYPROBLEM}. Some questions concerning the possible values of density measures were also stated in \cite{VANDOUWEN}. Several expressions of extreme values of density measures are known, see \cite{SLEZZIMDENSRANGE}. In this paper we continue in the study of these extremal values and we find new possibilities how to express them.

Every finitely additive measure on $\NN$ induces a positive continuous linear functional on $\ell_\infty$ and by restricting such functional to characteristic sequences $\chi_A$ we obtain a finitely additive measure. This correspondence between measures and functionals is described in more detail in Section \ref{SECTMEASFUNCT}. Often it is easier to study the corresponding functionals instead of measures. We show that the functionals corresponding to density measures are precisely the positive functionals extending Ces\`{a}ro mean. Thus we are also interested in the extremal values of these functionals.

Another natural question pertaining to the density measures (and the corresponding functionals) is whether we can obtain them in some way from some set of simpler measures (simpler functionals). One possibility how to do this was proposed in \cite{LAUWERS}. Unfortunately, as shown in \cite{SLEZZIMDENSLEVY}, the expression given in \cite[p.49]{LAUWERS} does not entail all density measures.

Using the fact that we know the range of possible values of density measures together with results of \cite{JERISONCONVEX} we obtain some sets of positive functionals which give the set of positive functionals extending Ces\`{a}ro mean as their closed convex hulls in $\ell_\infty^*$ with weak${}^*$-topology. (In the other words, we find a set of functionals such that taking all pointwise limits of their convex combinations yields the set of all positive functionals extending Ces\`{a}ro mean.) This also gives a corresponding results for density measures.

The paper is organized as follows:
Section \ref{SECTPRELIM} contains a version of Krein-Milman theorem which will be needed later.
In Section \ref{SECTMEASFUNCT} we briefly describe the correspondence between the finitely additive measures and the positive functionals from $\ell_\infty^*$. We show that in this way  density measures can be viewed as the positive functionals extending Ces\`{a}ro mean.

In Section \ref{SECTDENSMEAS} we give  four different expressions of the range of density measures. We start by stating the result that these four values are equal to each other, recapitulating which of the equalities are already known. Then we show the remaining equality.

Section \ref{SECTFUNCTIONALS}  deals with analogous problem for functionals. Here we only give two different expressions of maximal possible value of a functional extending Ces\`{a}ro mean. The purpose of this section is to prepare results needed in the last section, where we use them to derive some facts about extreme points of the set of density measures and corresponding functionals.

In Section \ref{SECTEXTPOINTS} we describe a set of functionals such that the closed convex hull of this set is the set of all positive functionals extending the Ces\`{a}ro mean. We give consequences of this fact for the density measures.
To get such set of functionals we can use the expressions of the extreme values obtained in Section \ref{SECTFUNCTIONALS}. It suffices to find some functionals, which attain these values. In this way it is relatively easy to find a set generating all functionals with this property (see Remark \ref{REMFORFREE}). But we then show that the description of the density measures and the corresponding functionals can be further simplified and we show that there is a smaller set of functionals, which have simpler form, and they still generate all positive functionals extending Ces\`{a}ro mean.
From this result we get in Corollary \ref{CORCONVHULL} an analogous result for density measures. This can be considered as the main result of this paper, here we finally get a results about extreme points of the sets of all density measures.

\section{Preliminaries}\label{SECTPRELIM}

By $\NN$ we denote the set $\NN=\{1,2,3,\dots\}$ of all positive integers.

\subsection{Extreme points of subsets of $X^*$}

We will need the following result, which is a special case of \cite[Theorem 1]{JERISONCONVEX} applied to the space $X^*$ with the weak${}^*$ topology. The result from \cite{JERISONCONVEX} can be considered as a version of Krein-Milman theorem.
\begin{PROP}\label{PROPJERISON}
Let $X$ be a linear normed space and $C$ be a subset of $X^*$ which is convex and compact in the weak${}^*$-topology. Let $S\subseteq C$. The following conditions are equivalent:
\begin{enumerate}
\enu
  \item
  \begin{equation*}\label{EQJERISON}
    \sup_{\vp\in S} \vp(x) = \sup_{\vp\in C}\vp(x)
  \end{equation*}
    holds for each $x\in X$;
  \item $C=\cco(S)$, i.e., $C$ is the closed convex hull of $S$;
  \item the closure $\ol S$ of the set $S$ contains all extreme points of $C$.
\end{enumerate}
\end{PROP}

\section{Finitely additive measures and $\ell_\infty^*$}\label{SECTMEASFUNCT}

We say that a function $\Zobr\mu{\powerset{\NN}}{\RR}$ is a \emph{finitely additive measure on $\NN$} if $\mu(A\cup B)=\mu(A)+\mu(B)$ for any disjoint sets $A,B\subseteq\NN$. (Notice that, at this point, we allow the values of $\mu$ to be negative.)

There is a very natural correspondence between the finitely additive measures and linear continuous functionals $\Zobr f{\ell_\infty}{\RR}$. For  $f\in\ell_\infty^*$ we can obtain a measure by putting
$$\mu(A)=f(\chi_A).$$
The process of obtaining a functional from a given measure is similar to definition of Riemann integral. It uses the fact that any bounded sequence can be uniformly approximated by step sequences. (By a \emph{step sequence} we mean a sequence of the form $\sum_{i=1}^n c_i\chi_{A_i}$ for some $\Ldots c1n \in\RR$ and $\Ldots A1n\subseteq\NN$, i.e., a finite linear combination of characteristic sequences.)

More details about this construction can be found, for example, in \cite[Theorem 16.7]{CAROTHERS}, \cite[p.50, Example 1.19]{MORRISON}, \cite[Section 3]{VANDOUWEN}. Many texts in functional analysis provide also a more general version of this result dealing with dual of $\mc L_\infty(X,\mu)$.

It is relatively easy to see that positive functionals correspond to positive finitely additive measures and positive functionals such that $\norm f=1$ correspond to positive finitely additive probability measures.

From now on we will say briefly \emph{measure} instead of finitely additive positive probability measure.

We will study the measures which extend asymptotic density.
\begin{DEF}
For a set $A\subseteq\NN$ the \emph{upper} and \emph{lower asymptotic density} is defined by
$$\ol d(A)=\limsup_{n\to\infty} \frac{A(n)}n \qquad \text{and} \qquad \ul d(A)=\liminf_{n\to\infty} \frac{A(n)}n,$$
where $A(n)=\abs{A\cap\{1,2,\dots,n\}}$.

If $\ul d(A)=\ol d(A)$ then this common value is denoted by $d(A)$ and it is called the \emph{asymptotic density} of the set $A$.

We will denote the class of all subsets of $\NN$ possessing asymptotic density by $\DD$.

A finitely additive measure $\Zobr\mu{\powerset{\NN}}{[0,1]}$ is called a \emph{density measure} if
$$\mu(A)=d(A)$$
for every $A\in\DD$.

The set of all density measures will be denoted by $\DM$.
\end{DEF}

In the other words, density measures are precisely the measures which extend asymptotic density.
Density measures have been studied, for example, in \cite{BFPR,MAHARAM,SALTIJD,VANDOUWEN,SLEZZIMDENSLEVY,SLEZZIMDENSRANGE}.

Sometimes it is easier to work with the corresponding functionals instead of measures. Since we work with density measures, we need to know what class of functionals corresponds to such measures. We will show in Theorem \ref{THMDENSCES} that they are precisely the positive functionals which extend Ces\`{a}ro mean.

\begin{DEF}
Let $x\in\ell_\infty$. If the limit
$$C(x)=\limti n \frac{x_1+\dots+x_n}n$$
exists, then it is called the \emph{Ces\`{a}ro  mean} of the sequence $x$. We will denote the set of all bounded sequences that have Ces\`{a}ro mean by $\Ces$.

The set of all positive
functionals
$f\in\ell_\infty^*$ which \emph{extend Ces\`{a}ro mean,} i.e., which fulfill
$$f|_{\Ces}=C$$
will be denoted by $\Cesex$.
\end{DEF}

In the following proof we will need the notion of L\'evy group $\mc G$. The L\'evy group $\mc G$ consists of all permutations $\Zobr\pi{\NN}{\NN}$ such that
$$\limti n \frac{\abs{\{k; k \le n < \pi(k)\}}}n=0.$$
Several equivalent characterizations and some other facts about $\mc G$ can be found in \cite{BLUMLINGER}, \cite{BLUMOBA} or \cite{SLEZZIMDENSLEVY}.

\begin{THM}\label{THMDENSCES}
Let $\mu$ be a measure and $f\in\ell_\infty^*$ be the corresponding functional. The measure $\mu$  is a density measure if and only if $f$ extends Ces\`{a}ro mean.
\end{THM}

\begin{proof}
It is obvious that if $f$ extends Ces\`{a}ro mean then $\mu$ is a density measure, since we have $d(A)=C(\chi_A)$ whenever $A\in\DD$.

To prove the opposite direction we need
a characterization of $\DM$ and $\Cesex$
using invariance under
permutations from $\mc G$.

It is shown in \cite[Theorem 2.6]{SLEZZIMDENSLEVY} that density measures are precisely the $\mc G$-invariant measures.
By \cite[Theorem 2]{BLUMOBA} a positive functional with $\norm f=1$ extends Ces\`{a}ro mean if and only if it is $\mc G$-invariant.

So if $\mu$ is a density measure, then the corresponding functional is positive and $\norm f=1$. Since $\mu$ is $\mc G$-invariant, we also get that $f$ is $\mc G$-invariant for step sequences, simply because $$f(\pi(\chi_A))=f(\chi_{\inv\pi(A)})=\mu(\inv\pi(A))=\mu(A)=f(\chi_A)$$
and the same is true for any finite linear combination $\sum\limits_{i=1}^n \chi_{A_i}$ by linearity.

Using the fact that step sequences are dense in $\ell_\infty$ we can extend the $\mc G$-invariance to all sequences.
It suffices to notice that if we have step sequences $s_n=\sum\limits_{i\in F_n} \chi_{A_i}$ (where each $F_n$ is finite) and $s_n\to x$, then we have, for any $\pi\in\mc G$,
$$f(\pi(s_n))=f(s_n)$$
which implies $f(\pi(x))=f(x)$.
\end{proof}

\section{Extremal values of density measures}\label{SECTDENSMEAS}

For a given subset $A\subseteq\NN$, we are interested in the values
$$\lld(A)=\inf\{\mu(A); \mu\in\DM\} \qquad\text{and}\qquad \uud(A)=\sup\{\mu(A); \mu\in\DM\}.$$
In this section we will mention several equivalent expressions of $\uud(A)$ and $\lld(A)$.

For any $A\subseteq\NN$ and $\alpha\in\RR$ we define
$$A_\alpha(n)=\sum_{\substack{k\in A\\k\le n}}k^\alpha.$$
Notice that $A_0(n)=A(n)$.

It is relatively easy to see that $\NN_\alpha(n)\sim\frac{n^{\alpha+1}}{\alpha+1}$ for $\alpha>-1$. This observation is often useful when working with functions similar to $\udi(A)$, $\ldi(A)$ defined below.

In addition to $\uud(A)$ and $\lld(A)$ we will also use the following quantities.
\begin{align*}
\du(A)&=\inf\{d(B); B\supseteq A; B\in\DD\}\\
\udi(A)&=\limti\alpha \limsup_{n\to\infty} \frac{A_\alpha(n)}{\NN_\alpha(n)}\\
\udp(A)&=\lim_{\theta\to1^-} \limsup_{n\to\infty}  \frac{A(n)-A(\theta n)}{n-\theta n}
\end{align*}

We should show that the limits in the definitions of $\udi(A)$ and $\udp(A)$ really exist.
The existence of the limit appearing in the definition of $\udp(A)$ is shown in \cite{POLYA}.
For the existence of the limit used in the definition of $\udi(A)$, see Remark \ref{REMUDIEXISTS}.

We also define the lower versions of the last three densities $\dl(A)=\sup\{d(B); B\subseteq A; B\in\DD\}$, $\ldi(A)=\limti\alpha \liminf\limits_{n\to\infty} \frac{A_\alpha(n)}{\NN_\alpha(n)}$ and $\ldp(A)=\lim\limits_{\theta\to1^-} \liminf\limits_{n\to\infty}  \frac{A(n)-A(\theta n)}{n-\theta n}$.

There is an obvious correspondence between the lower and upper densities, $\uud(\NN\sm A)=1-\lld(A)$ and the analogous equalities are true for
$\du(A)$, $\udi(A)$ and $\udp(A)$.

The values $\udp(A)$ and $\ldp(A)$ were studied in a slightly more general setting in \cite{POLYA}. They are sometimes called (upper and lower) P\'olya density, see e.g. \cite{GREDENSSURV}.

The densities $\udi(A)$ and $\ldi(A)$ are related to $\alpha$-densities, which were studied
in \cite{FUCHSGA,GAGM}. The upper $\alpha$-density is defined as
$$\uda{\alpha} (A) = \limsup_{n\to\infty} \frac{A_\alpha(n)}{\NN_\alpha(n)},$$
the definition of $\lda{\alpha}(A)$ is analogous.

We will need the following facts, which follow from \cite[Theorem 3, Corollary 1]{RAJAGOPAL}.
\begin{THM}\label{THMRAJNEW}
Let
\begin{enumerate}
\enu
  \item If $-1\le\alpha\le\beta$ and then
    $$\lda{\beta}(A) \leq \lda{\alpha}(A) \leq \uda{\alpha}(A) \leq \uda{\beta}(A)$$
  \item If $\alpha>0$ and $A\in \DD$, then
    $$\uda{\alpha}(A) = d(A).$$
\end{enumerate}
\end{THM}

\begin{REM}\label{REMUDIEXISTS}
Note that from Theorem \ref{THMRAJNEW} we immediately get that the limit
$$\udi(A)=\limti\alpha \uda\alpha(A)$$
exists.
\end{REM}

Now we can formulate the main result of this section, in which we obtain several expressions for the extremal values of density measures.
\begin{THM}\label{THM4DENS}
For any $A\subseteq\NN$ we have
$$\uud (A)= \du(A)=\udi(A)=\udp(A)$$
and $\lld(A)=\dl(A)=\ldi(A)=\ldp(A).$
\end{THM}

Of course, it suffices to show this only for the upper densities or only for lower densities, because of the aforementioned correspondence.

In fact, with the exception of the part about $\udi(A)$ and $\ldi(A)$ the above theorem is only a summarization of known results.
The equality $\uud(A)=\du(A)$ was shown in \cite[Theorem 3]{SLEZZIMDENSRANGE}.  The equality $\uud(A)=\udp(A)$ can be obtained from (more general) result \cite[Satz VIII]{POLYA}.

The missing equality between $\udi(A)$ and the remaining expressions will be shown in Lemma \ref{LMDODATOK}. This answers \cite[Problem 1]{SLEZZIMDENSRANGE}.

Let us start by summarizing some results on $\udi(A)$ and $\uud(A)$ which will be needed in the proof.

\begin{LM}\label{LMSLEZZIM}
Let $A,B\subseteq\NN$.
\begin{enumerate}
\enu
  \item
        $$0 \le \lld(A) \leq \ldi(A) \leq \udi(A) \leq \uud(A)\le1$$
  \item For any $B$ there exists $A\subseteq B$ such that $A\in\DD$ and $d(A)=\lld(B)$.
  \item If $A\cap B=\emps$ and $A\in\DD$, then $$\lld(A\cup B)=d(A)+\lld(B).$$
  \item If $A\cap B=\emps$ and $A\in\DD$ then $$\ldi(A\cup B)=d(A)+\ldi(B).$$
\end{enumerate}
\end{LM}

\begin{proof}
The part (i) is shown in \cite[Corollary 6]{SLEZZIMDENSRANGE}. The second part is \cite[Proposition 2]{SLEZZIMDENSRANGE} and the third part is
\cite[Proposition 1]{SLEZZIMDENSRANGE}.

(iv) From the existence of the limit $d(A)=\limti n\frac{A(n)}{n}=\limti n \frac{\sum_{i=1}^n \chi_A(n)}n$ we get, for any $\alpha>0$,
$d_\alpha(A)=\limti n \frac{A_\alpha(n)}{\NN_\alpha(n)} =d(A)$ from Theorem \ref{THMRAJNEW}. Therefore
$$\lda{\alpha} (A\cup B) = \liminf_{n\to\infty} \frac{A_\alpha(n)+B_\alpha(n)}{\NN_\alpha(n)} =
\limti n \frac{A_\alpha(n)}{\NN_\alpha(n)} + \liminf_{n\to\infty} \frac{B_\alpha(n)}{\NN_\alpha(n)} = d(A) + \lda{\alpha}(B).$$
Taking limits $\alpha\to\infty$ on both sides gives $\ldi(A\cup B)=d(A)+\ldi(B)$.
\end{proof}

\begin{LM}\label{LMDODATOK}
For any set $A\subseteq\NN$ we have
$$\lld(A)=\ld_\infty(A).$$
\end{LM}

\begin{proof}
We already know $0 \le \lld(A) \leq \ldi(A) $ from Lemma \ref{LMSLEZZIM}, so it suffices to show the inequality $\ldi(A)\le\lld(A)$ for all $A\subseteq\NN$.

First let us assume that $\lld(A)=0$. We want to show that $\lda\infty(A)=0$.

We have
$$\lim_{\theta \to 1^-} \liminf_{n \to \infty} \frac{A(n) - A(\theta n)}{n - \theta n}=0.$$
Fix $\ve>0$. Then there exists $\theta_0<1$ such that
$$\liminf_{n \to \infty} \frac{A(n) - A(\theta n)}{n - \theta n}<\ve$$
whenever $\theta_0<\theta<1$. This implies that
\begin{equation}\label{EQTHETA}
  \frac{A(n) - A(\theta n)}{n - \theta n} < 2\ve
\end{equation}
for infinitely many $n$'s.

If $n$ fulfills (\ref{EQTHETA}) then
\begin{gather*}
A_\alpha(n) \leq A_\alpha(\theta n) + n^\alpha (A(n) - A(\theta n)) < A_\alpha(\theta n) +
n^{\alpha+1}(1-\theta) 2 \ve\\
\frac{A_\alpha(n)}{n^{\alpha+1}} \leq \frac{A_\alpha(\theta n)}{(\theta
n)^{\alpha+1}}\cdot\theta^{\alpha+1} +
(1-\theta)2\ve\\
\frac{A_\alpha(n)}{n^{\alpha+1}} \leq \frac{\NN_\alpha(\theta n)(\alpha+1)}{(\theta n)^{\alpha+1}} \cdot \frac{\theta^{\alpha+1}}{\alpha+1} + (1-\theta)2\ve
\end{gather*}

Since the above inequality holds for infinitely many $n$'s we get
\begin{equation}\label{INEQTHETA}
\lda{\alpha}(A) = \liminf_{n\to\infty} \frac{A_\alpha(n)(\alpha+1)}{n^{\alpha+1}} \leq \theta^{\alpha+1} +
(1-\theta)2\ve(\alpha+1)
\end{equation}
for any $\theta\in(\theta_0,1)$ and any $\alpha>0$.

Let us assume that, moreover,
\begin{equation}\label{EQTHETAALPHA}
    (1-\theta)(\alpha+1)=\frac1{\sqrt{\ve}}.
\end{equation}

Using \eqref{EQTHETAALPHA} we get from \eqref{INEQTHETA}
\begin{equation}\label{EQALPEPS}
    \lda{\alpha}(A) \leq \left(1-\frac1{(\alpha+1)\sqrt\ve}\right)^{\alpha+1} + 2\sqrt\ve.
\end{equation}
This inequality is valid for any $\alpha$ and
$\theta\in(\theta_0,1)$ that fulfill \eqref{EQTHETAALPHA}.

Now, if $\alpha\to\infty$ (note that this means $\theta\to1^-$,
hence we can always find $\theta\in(\theta_0,1)$ such that
\eqref{EQTHETAALPHA} holds), we get
\begin{equation}\label{EQINFEPS}
    \lda{\infty}(A) \leq e^{-\frac1{\sqrt\ve}} + 2\sqrt\ve.
\end{equation}
As the RHS tends to $0$ for $\ve\to0^+$ and $\ve>0$ can be
chosen arbitrarily, we finally get
$$\lda{\infty}(A)=0.$$

If we combine the above with the inequality $0 \le \lld(A) \leq \ldi(A)$, we have so far proved $\lld(A)=\ld_\infty(A)$ for the case that some of these values is zero.

By Lemma \ref{LMSLEZZIM}(ii) we know that
$\lld(A)=d(B)$
for some $B\subseteq A$, $B\in\DD$. For this set we have $\lld(A\sm B)=0$. Using other parts of Lemma \ref{LMSLEZZIM} we get that
$$\lld(A)=\lld(B\cup A\sm B) = d(B)+\lld(A\sm B)= d(B)+\ldi(A\sm B)=\ldi(A).$$
\end{proof}

From Theorem \ref{THM4DENS} and Theorem \ref{THMRAJNEW} we immediately have $$\lld(A) \le \ul d(A) \le \ol d(A) \le \uud(A).$$
Let us mention that these inequalities may be strict, as shown, for example, in  \cite[Exaple 3.4]{SLEZZIMDENSLEVY}, \cite[Sectoin 2]{SLEZZIMDENSRANGE}.

\section{Extremal values of positive functionals extending Ces\`{a}ro mean}\label{SECTFUNCTIONALS}

Next we turn to study the set $\Cesex$ of all positive functionals extending Ces\`{a}ro mean. From Theorem \ref{THMDENSCES} we already know that they correspond to density measures. Sometimes working with functionals instead of measures can be more convenient.

Again we want to find out the maximal and minimal possible value of such functionals for a given bounded sequence $x$.
We will study two functions which naturally correspond to $\uud$ and $\udp$. (It would be possible to define also analogous of the remaining densities from Theorem \ref{THM4DENS}.)

These results will be used in the following section to prove Corollary \ref{CORCONVHULL} describing a set containing all extreme points of the set of all density measures.

We will use the following sublinear functions defined on $\ell_\infty$
\begin{align*}
f_C(x)&=\sup\{f(x); f\in\Cesex\}\\
t(x)&=\lim_{\theta\to1^-} \limsup_{n\to\infty} \frac{\sum x_i\chi_{(\theta n,n]}(i)}{n(1-\theta)}
\end{align*}

The existence of the limit used in the definition of $t(x)$ is shown in Remark \ref{REMTEXISTS}.

\begin{REM}\label{REMDENSFUN}
Note that for $A\subseteq\NN$ we have $f_C(\chi_A)=\uud(A)$, $t(\chi_A)=\udp(A)$.
\end{REM}

We will introduce some notation which will be useful when proving some facts about the function $t(x)$.

Let us denote
\begin{align*}
\Te_{\theta,n}(x)&=\frac{\sum x_i\chi_{(\theta n,n]}(i)}{n(1-\theta)}\\
\Te_\theta(x)&=\limsup_{n\to\infty} \Te_{\theta,n}(x)
\end{align*}
When the sequence $x$ will be clear from the context, we will just write $\Te_{\theta,n}$ and $\Te_\theta$ instead.

This means that $t(x)=\lim\limits_{\theta\to1^-} \Te_\theta(x)$.

Now we can formulate and prove the following lemma, which will be useful in the proof of Proposition \ref{PROPTX}.

\begin{LM}\label{LMSLAVO2}
Let $x\in\intrv01^{\NN}$. Then there exists $\tilde x\in\{0,1\}^{\NN}$ such that $\tilde x-x\in\Ces$ and
\begin{gather*}
C(\tilde x-x)=0\\
f_C(x)=f_C(\tilde x)\\
\Te_{\theta}(x)=\Te_{\theta}(\tilde x) \text{ and } t(x)=t(\tilde x)
\end{gather*}
\end{LM}

\begin{proof}
Define $\tilde x_n$ by
$$\Pdots{\tilde x}1n = \dcc{\Pdots x1n}.$$
Then we have
$$\Pdots{\tilde x}1n \le \Pdots x1n \le \Pdots{\tilde x}1n+1,$$
which implies $C(\tilde x-x)=0$. From this we immediately get $f_C(x)=f_C(\tilde x)$.

We also have
$$\sum x_i\chi_{(\theta n,n]}(i) = \sum_{i=1}^n x_i - \sum_{i=1}^{\dcc{\theta n}} x_i,$$
which implies $\abs{\sum x_i\chi_{(\theta n,n]}(i) - \sum \tilde x_i\chi_{(\theta n,n]}(i)}\le 2$ and, consequently,
$\Te_{\theta}(x)=\Te_{\theta}(\tilde x)$ and $t(x)=t(\tilde x)$.
\end{proof}

\subsection{Basic properties of $f_C(x)$}

Let us mention some basic properties of the function $f_C(X)$. Proofs of these facts are easy, so they will be omitted.

\begin{LM}\label{LMFC}
\begin{align*}
(\forall x,y\in\ell_\infty)\,\,& f_C(x+y)\le f_C(x)+f_C(y)\\
(\forall x\in\ell_\infty)(\forall c>0)\,\,& f_C(cx)=cf_C(x)\\
(\forall x\in\ell_\infty)(\forall y\in\Ces)\,\,& f_C(x+y)=f_C(x)+C(y)\\
(\forall x\in\Ces)\,\,& f_C(x)=C(x)\\
(\forall x,y\in\ell_\infty)\,\,& x-y\in\Ces \land C(x-y)=0 \Ra f_C(x)=f_C(y)\\
(\forall x,y)\,\,& x\le y \Ra f_C(x)\le f_C(y)
\end{align*}
\end{LM}

\subsection{Function $t(x)$}

\begin{REM}\label{REMTEXISTS}
From \cite{POLYA} we know that the limit in the definition of $t(x)$ exists for sequences from $\{0,1\}^{\NN}$.
Using Lemma \ref{LMSLAVO2} we get the existence of this limit for any sequences in $\intrv01^{\NN}$.

It is relatively easy to show that $\Te_\theta(cx)=c\Te_\theta(x)$ for $c>0$ and $\Te_\theta(x+\ol c)=c+\Te_\theta(x)$ for any constant $c$. Once we know this, we get the existence of this limit for all bounded sequences.

In fact, the proof from \cite{POLYA} could be modified so that it works for any bounded sequence.
\end{REM}

Now we can prove the main result of this section.

\begin{PROP}\label{PROPTX}
For any $x\in\ell_\infty$ we have
$$t(x)=f_C(x).$$
\end{PROP}

\begin{proof}
We know from Theorem \ref{THM4DENS} that $t(\chi_A)=f_C(\chi_A)$ for any $A\subseteq\NN$. This yields $t(\tilde x)=f_C(\tilde x)$ for any $\tilde x\in\{0,1\}^{\NN}$ (see Remark \ref{REMDENSFUN}).

Now from Lemma \ref{LMSLAVO2} we have for any $x\in\intrv01^{\NN}$ that there is a sequence $\tilde x \in\{0,1\}^{\NN}$ such that
$$t(x)=t(\tilde x)=f_C(\tilde x)=f_C(x).$$
This proves the claim for any $x\in\intrv01^{\NN}$.

We can extend the validity of the above equality from $x\in\intrv01^{\NN}$ to $x\in\ell_\infty$ using the facts that $\Te_\theta(cx)=c\Te_\theta(x)$ for $c>0$ and $\Te_\theta(x+\ol c)=c+\Te_\theta(x)$ for any constant $c$, together with analogous claims for $f_C(x)$ (see Lemma \ref{LMFC}).
\end{proof}

\section{Extreme points}\label{SECTEXTPOINTS}

We have found the maximal value of $f(x)$ for functionals $f\in\Cesex$.
We will use Proposition \ref{PROPJERISON} to obtain some sets of functionals such that their closed convex hull is $\Cesex$.
This gives, in a sense, a description of the whole space $\Cesex$ using some simpler functionals. Namely we can obtain all functionals by taking the convex hull of this set and then the closure in the weak${}^*$-topology. (Closure in the weak${}^*$-topology means taking all limits of nets of functionals from this set, which converge pointwise.)

We could obtain some similar sets of functionals with this property from the results we have already shown so far, see Remark \ref{REMFORFREE}. The main purpose of this section is to obtain a smaller and simpler set which generates all functionals extending Ces\`{a}ro mean.
Using the correspondence between measures and functionals this yields a set of density measures, which generate all density measures, see Corollary \ref{CORCONVHULL}.

In the proof of Theorem \ref{THMSLAVOSUPFG}, which is one of the main results of this section, we will be working with some fixed sequence $x=\seq xn \in \intrv01^{\NN}$. It suffices to show this result for sequences with values in the interval $\intrv01$, since the result can be easily extended to all bounded sequences by scaling and adding constant sequence.

We can use $\Te_{\theta,n}$ from the definition of the function $t(x)$ not only for integers, but for positive real numbers, too.
For any real number $r>0$ and $\theta\in(0,1)$ we define:
\begin{align*}
\Sui{\theta}r(x)&=\sum_{i\in\NN} x_i\chi_{\intrvr{\theta r}r}(i)\\
\Te_{\theta,r}(x)&=\frac{\Sui{\theta}r(x)}{r(1-\theta)}\\
\Te_\theta(x)&=\limsup_{r\to\infty} \Te_{\theta,r}(x)\\
t(x)&=\lim\limits_{\theta\to1^-}\Te_\theta(x)
\end{align*}
We will often write just $\Sui{\theta}r$, $\Te_{\theta,r}$ and $\Te_\theta$ instead.

The resulting function $t(x)$ will be the same in both cases, whether we use integers or real numbers. (We will need to use both possibilities.)
To see that this is indeed the case, notice that if $\abs{r-r'}<1$ (e.g., if $r'=\dcc r$ or $r'=\hcc r$) then also $\abs{\theta r-\theta r'}<1$ and consequently
$$\abs{\Sui{\theta}r-\Sui{\theta}{r'}}\le 2.$$
Hence $$\limti r \left(\frac{\Sui{\theta}r}r-\frac{\Sui{\theta}{r'}}{r'}\right)=0$$
holds for any $\theta\in(0,1)$.

We will use the notion of the limit of a real sequence along an ultrafilter.
\begin{DEF}
If $\seq xn$ is a bounded real sequence and $\FF$ is an ultrafilter on $\NN$, then there exists a unique real number $L$ such that
$$\{n\in\NN; \abs{x_n-L}<\ve\}\in\FF$$
for each $\ve>0$. The number $L$ is called the $\FF$-limit of the sequence $\seq xn$ and it is denoted by $L=\Flim x_n$.
\end{DEF}
The most important property of $\FF$-limit is that it exists for every bounded sequence. For some other useful facts about $\FF$-limits we refer the reader to
\cite{ALEKSGLEBGORD}, \cite[8.23--8.26]{BALSTE}, \cite[Sections 2.3 and 4.5]{DIXMIERGT}, \cite[Section 11.2]{HRJECH}, \cite[Problem 17.19]{KOMJTOT}, \cite{KMSS}.

The symbols $\beta\NN$ and $\beta\NN^*$ denote the set of all ultrafilters and the set of all free ultrafilters on $\NN$, respectively.

\begin{THM}\label{THMSLAVOSUPFG}
For any sequence $\theta_k\in(0,1)$ such that $\theta_k\to1^-$, any $x\in\ell_\infty$ and any free ultrafilter $\FF$ we have
$$t(x)=\sup_{\GG\in\beta\NN^*} \Flim_k \Glim_n \Te_{\theta_k,n}(x).$$
\end{THM}

\begin{proof}
Fix a sequence $x=(x_k)\in \intrv01^{\NN}$.

Choose any $L<H<t(x)$.

By the definition of $t(x)$ there exists an $\thenul\in(0,1)$ such that
$$\Te_\theta=\limsupti r \Te_{\theta,r}>H$$
for each $\theta\ge\thenul.$

This implies that the set $$\{n\in\NN; \Te_{\thenul,n} \ge H\}$$
is infinite.

Inductively we can choose an infinite subset $A=\{n_1<n_2<\dots\}$ such that $\frac{n_k}{n_{k+1}}<\thenul$.
We have
$$n\in A \qquad \Ra \qquad \Te_{\thenul,n}\ge H.$$

Using this set we can define a function $\Zobr\vp{\intrvl{\thenul}1}{\intrv01}$ by
$$\vp(\theta)=\liminf_{n\in A} \Te_{\theta,n}$$

This function is continuous. (The proof is postponed to Lemma \ref{LMVPCONT}.)

\underline{Case 1.} First we suppose that $(\forall \theta\in\intrvl{\thenul}1) (\vp(\theta)\ge H)$.
This means that
$$\liminf_{n\in A} \Te_{\theta,n}\ge H.$$ For every such $\theta$ and for any free ultrafilter $\GG$ containing the infinite set $A$ we get $$\Glim_n \Te_{\theta,n}\ge H\ge L.$$

\underline{Case 2.} If $\vp$ does not fulfill the above property, then there exists $\theo>\thenul$ such that $L<\vp(\theo)<H$. Let us denote $M=\vp(\theo)$ and
$$\thed=\inf\{\theta\in\intrvl{\thenul}1; \vp(\theta)\le M\}.$$
We have
$$\vp(\thed)=M$$
(by the continuity of $\vp$) and the definition of $M$ implies
$$\vp(\theta)>M$$
whenever $\thenul\le\theta<\thed$. (See Figure \ref{FIGFUNKCIA}.)

\begin{figure}
\begin{center}
\includegraphics{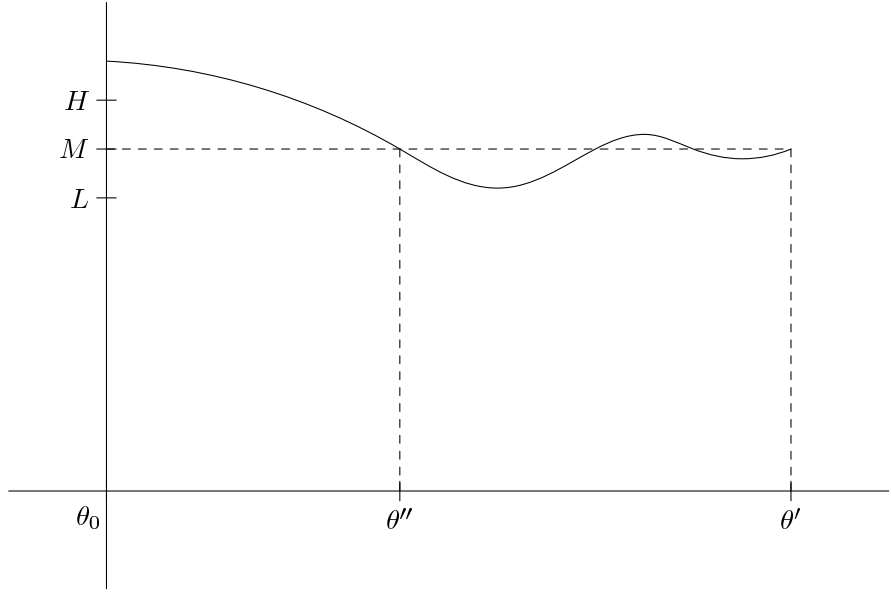}
\end{center}
\caption{Illustration of $\theo$ and $\thed$ from the proof of Theorem \ref{THMSLAVOSUPFG}}\label{FIGFUNKCIA}
\end{figure}

Since $$M=\vp(\thed)=\liminf_{n\in A}\Te_{\thed,n},$$
there exists an infinite set $A'\subseteq A$ such that
\begin{gather*}
\lim_{n\in A'} \Te_{\thed,n} =\lim_{n\in A'}\frac{\Sui{\thed}n}{n(1-\thed)}=M\\
\intertext{or equivalently}
\lim_{n\in A'}\frac{\Sui{\thed}n}{n}=M(1-\thed)
\end{gather*}

For every $n\in A'$ and $\theta$ such that $\frac{\thenul}{\thed}\le\theta<1$
we get
\begin{gather}
\notag \Sui{\theta\thed}n=\Sui{\theta}{\thed{n}}+\Sui{\thed}n,\\
\frac{\Sui{\theta\thed}n}n=\frac{\Sui{\theta}{\thed{n}}}n+\frac{\Sui{\thed}n}n, \label{EQZLOM}
\end{gather}
(since $\intrvl{\theta\thed n}n = \intrvl{\theta\thed n}{\thed n} \cup \intrvl{\thed n}n$; see Figure \ref{FIGINTERVALY}.)

\begin{figure}
\begin{center}
\includegraphics{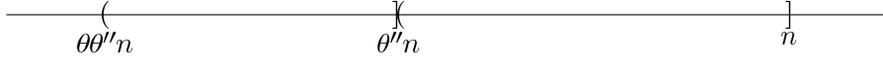}
\end{center}
\caption{Division of the interval $\intrvr{\theta\thed n}n$}\label{FIGINTERVALY}
\end{figure}

The condition $\thenul \le \theta\thed < \thed$ implies $\vp(\theta\thed)\ge M$, i.e.,
\begin{align*}
\liminf_{n\in A'} \Te_{\theta\thed,n}  &\ge \liminf_{n\in A} \Te_{\theta\thed,n} \ge M\\
\liminf_{n\in A'} \frac{\Sui{\theta\thed}n}{n} &\ge M(1-\theta\thed)
\end{align*}

From \eqref{EQZLOM}
we get $\frac{\Sui{\theta}{\thed{n}}}n = \frac{\Sui{\theta\thed}n}n-\frac{\Sui{\thed}n}n$ and thus
\begin{gather*}
\liminf_{n\in A'} \frac{\Sui{\theta}{\thed{n}}}n = \liminf_{n\in A'} \frac{\Sui{\theta\thed}n}n-\lim_{n\in A'} \frac{\Sui{\thed}n}n \ge
M(1-\theta\thed)-M(1-\thed)=M\thed(1-\theta)\\
\liminf_{n\in A'} \frac{\Sui{\theta}{\thed{n}}}{\thed n(1-\theta)} \ge M
\end{gather*}

Now denote $$B=\{\thed a; a\in A'\}.$$ (Note that $B\subseteq\RR$.)

Since $B$ is precisely the set of all real numbers of the form $r=\thed n$ for $n\in A'$, this is equivalent to
\begin{align*}
\liminf_{r\in B} \frac{\Sui{\theta}r}{r(1-\theta)} &\ge M,\\
\liminf_{r\in B} \Te_{\theta,r} &\ge M.
\end{align*}

Now put $$\dcc B=\{\dcc b; b\in B\}.$$ For this set we have
$$\liminf_{n\in \dcc B} \Te_{\theta,n} \ge M.$$
(We have chosen the set $A$ in such way that for any two elements $m_1<m_2$ in $A$ we have $\frac{m_1}{m_2}<\thenul$. This implies that $\thed m_1 < m_1 < \thenul m_2 <\thed m_2$. Hence for any two different elements $m_1, m_2\in A'$, the corresponding elements $\dcc{\thed m_1}$ and $\dcc{\thed m_2}$ of $B$ will be different.)

This implies that for every free ultrafilter $\GG\ni\dcc{B}$ the inequalities
$$\Glim \Te_{\theta,n} \ge M \ge L$$
hold for each $\theta\in\intrvl{\frac{\thenul}{\thed}}1$.

\underline{Conclusion.} In both cases we have shown that there exists a free ultrafilter $\GG$ such that
$$\Glim \Te_{\theta,n} \ge L$$
for each $\theta$ close enough to $1$. Thus for any choice of $\seq\theta{k}$ such that $\theta_k\to1^-$ and any free ultrafilter $\FF$
$$\Flim_k \Glim_n \Te_{\theta_k,n} \ge L.$$
Since $L$ can be chosen arbitrarily close to $t(x)$, we get
$$\Flim_k \Glim_n \Te_{\theta_k,n} \ge t(x).$$
\end{proof}

\begin{LM}\label{LMVPCONT}
The function $\vp(\theta)=\liminf\limits_{n\in A} \Te_{\theta,n}$ is continuous for any infinite set $A\subseteq\NN$.
\end{LM}

\begin{proof}
Note that we have for any $\theta\in(0,1)$, $r>0$ and $0<\delta<1-\theta$.
\begin{gather*}
\Sui{\theta}r=\sum x_i \chi_{\intrvr{(\theta+\delta)r}r} +\sum x_i \chi_{\intrvr{\theta r}{(\theta+\delta)r}}\\
\Sui{\theta}r=\Sui{(\theta+\delta)}r+\sum x_i \chi_{\intrvr{\theta r}{(\theta+\delta)r}}\\
\Te_{\theta,r}=\frac{\Sui{\theta}r}{r(1-\theta)}=\frac{\Sui{(\theta+\delta)}r}{r(1-\theta)}+\frac{\sum x_i \chi_{\intrvr{\theta r}{(\theta+\delta)r}}}{r(1-\theta)}\\
\Te_{\theta,r}=\Te_{\theta+\delta,r}\frac{1-(\theta+\delta)}{1-\theta}+\frac{\sum x_i \chi_{\intrvr{\theta r}{(\theta+\delta)r}}}{r(1-\theta)}\\
\Te_{\theta,r}=\Te_{\theta+\delta,r}-\frac{\delta}{1-\theta}\Te_{(\theta+\delta),r}+\frac{\sum x_i \chi_{\intrvr{\theta r}{(\theta+\delta)r}}}{r(1-\theta)}\\
\absl{\Te_{\theta,r}-\Te_{\theta+\delta,r}}\le \frac{\delta}{1-\theta} \norm{x} + \frac{\delta}{(1-\theta)} \norm{x}
\end{gather*}

From this we get
$$\absl{\liminf_{n\in A}\Te_{\theta+\delta,n}-\liminf_{n\in A}\Te_{\theta,n}}\le \frac{2\delta}{1-\theta} \norm{x},$$
which implies the continuity of $\varphi$.
\end{proof}

\begin{LM}\label{LMFGINCESEX}
Let $\FF,\GG\in\beta\NN^*$, and $\theta_k\to1^-$. Let us define $\Zobr f{\ell_\infty}{\RR}$ by
$$f(x)=\Flim_k \Glim_n \Te_{\theta_k,n}(x).$$
Then $f\in\Cesex$, i.e., $f$ is a positive linear functional on $\ell_\infty$ which extends Ces\`{a}ro mean.
\end{LM}

\begin{proof}
It is clear
that $f(x)$ defined above is a positive linear functional on $\ell_\infty$.

Now let $x\in\Ces$. Then we have $t(x)=\lim\limits_{\theta\to1^-} \limsup\limits_{n\to\infty} \Te_{\theta,n}(x)=C(x)$ and also
$C(x)=-C(-x)=-t(-x)=\lim\limits_{\theta\to1^-} \liminf\limits_{n\to\infty} \Te_{\theta,n}(x)$. For each $k$ we get
$$\liminf\limits_{n\to\infty} \Te_{\theta_k,n}(x) \le \Glim_n \Te_{\theta_k,n}(x) \le \limsup\limits_{n\to\infty} \Te_{\theta_k,n}(x)$$
which yields
$$C(x)\le \Flim_k \Glim_n \Te_{\theta_k,n}(x) \le C(x).$$
So we have $f(x)=C(x)$.
\end{proof}

\begin{THM}\label{THMCONVHULL}
Let $\theta_k\in(0,1)$ be a sequence such that $\theta_k\to1^-$ and $\FF$ be any free ultrafilter $\FF$. Let us define
$$S=\{f_{\GG}; \GG\in\beta\NN^*\}$$
where $\Zobr {f_{\GG}}{\ell_\infty}{\RR}$ is defined by
$$f_{\GG}(x)=\Flim_k \Glim_n \Te_{\theta_k,n}(x).$$
We consider $S$ as a subset of the space $\ell_\infty^*$ endowed with the weak${}^*$ topology.

Then $$\ccohull{S}=\Cesex,$$
i.e., the closed convex hull of $S$ is the set of all positive functionals on $\ell_\infty$ extending Ces\`{a}ro mean.
Or, equivalently, $\ol{S}$ contains all extreme points of $\Cesex$.
\end{THM}

\begin{proof}
The set $\Cesex$ is a subset of the unit ball of $\ell^*$. It is closed in the weak${}^*$ topology. Thus by Banach-Alaoglu theorem it is compact.

Lemma \ref{LMFGINCESEX} implies that each $\vp\in S$ is indeed an element of $\Cesex$ and by Theorem \ref{THMSLAVOSUPFG} we have
$$\sup_{\vp\in S} \vp(x)=\sup_{\vp\in\Cesex} \vp(x).$$
So from Proposition \ref{PROPJERISON} we get that $\ccohull{S}=\Cesex$ and $\ol S$ contains all extreme points of $\Cesex$.
\end{proof}

Of course, we can work with finitely additive measures instead of the corresponding functionals. The weak${}^*$ topology is then determined by the condition that a net $\mu_\sigma$ of measures converges to $\mu$ if and only if $\mu_\sigma(A)$ converges to $\mu(A)$ for each $A\subseteq\NN$.
\begin{COR}\label{CORCONVHULL}
Let $\theta_k\in(0,1)$ be a sequence such that $\theta_k\to1^-$ and $\FF$ be any free ultrafilter $\FF$. Let us define
$$S=\{\mu_{\GG}; \GG\in\beta\NN^*\}$$
where $\mu_{\GG}$ is defined by
$$\mu_{\GG}(A)=\Flim_k \Glim_n \Te_{\theta_k,n}(\chi_A).$$

Then $$\ccohull{S}=\DM,$$
i.e., the closed convex hull of $S$ is the set of all density measures.
Or, equivalently, $\ol{S}$ contains all extreme points of $\DM$.
\end{COR}

\begin{REM}\label{REMFORFREE}
It is easy to show that any functional of the form
$$f(x)=\Flim_k \Gklim_n \Te_{\theta_k,n}(x)$$
for $\FF,\GG_k\in\beta\NN^*$ is also a positive functional extending Ces\`{a}ro mean.

For any given $x$ and $(\theta_k)$ there exists a sequence $\GG_k$ of free ultrafilters such that $\limsup\limits_{n\to\infty} \Te_{\theta_k,n}=\Gklim_n \Te_{\theta_k,n}$.

From this we can immediately see that the set of all functionals of this form has the same properties as the set $S$ from Theorem \ref{THMCONVHULL}. But Theorem \ref{THMCONVHULL} is a stronger result than this, we do not need a sequence of ultrafilters, we have the same ultrafilter $\GG$ for each $k$.
\end{REM}

\bibliographystyle{plain}
\bibliography{martin}

\begin{thebibliography}{10}

\bibitem{ALEKSGLEBGORD}
M.~A. Alekseev, L.~{Yu.} Glebsky, and E.~I. Gordon.
\newblock On approximations of groups, group actions {and Hopf} algebras.
\newblock {\em Journal of Mathematical Sciences}, 107(5):4305--4332, 2001.

\bibitem{BALSTE}
B.~Balcar and P.~{\v{S}}t\v{e}p\'anek.
\newblock {\em Teorie mno\v{z}in}.
\newblock Academia, Praha, 2001.

\bibitem{BFPR}
A.~Blass, R.~Frankiewicz, G.~Plebanek, and C.~Ryll-Nardzewski.
\newblock A note on extensions of asymptotic density.
\newblock {\em Proc. Amer. Math. Soc.}, 129(11):3313--3320, 2001.

\bibitem{BLUMLINGER}
M.~Bl\"{u}mlinger.
\newblock L\'{e}vy group action and invariant measures on {$\beta \mathbb{N}$}.
\newblock {\em Trans. Amer. Math. Soc.}, 348(12):5087--5111, 1996.

\bibitem{BLUMOBA}
M.~Bl\"{u}mlinger and N.~Obata.
\newblock Permutations preserving {C}es{\'a}ro mean, densities of natural
  numbers and uniform distribution of sequences.
\newblock {\em Ann. Inst. Fourier}, 41:665--678, 1991.

\bibitem{CAROTHERS}
N.~L. Carothers.
\newblock {\em A Short Course on {B}anach Space Theory}.
\newblock Cambridge University Press, New York, 2004.
\newblock London Mathematical Society Student Texts 64.

\bibitem{DIXMIERGT}
J.~Dixmier.
\newblock {\em General Topology}.
\newblock Springer-Verlag, New York, 1984.
\newblock Undergraduate Texts in Mathematics.

\bibitem{FEYMAY}
M.~Fey.
\newblock May's theorem with an infinite population.
\newblock {\em Social Choice and Welfare}, 23:275--293, 2004.

\bibitem{FEYPROBLEM}
M.~Fey.
\newblock Problems {(Density measures)}.
\newblock {\em Tatra Mnt. Math. Publ.}, 31:177--181, 2005.

\bibitem{FUCHSGA}
A.~Fuchs and R.~Giuliano~Antonini.
\newblock Th\'{e}orie g\'{e}n\'{e}rale des densit\'{e}s.
\newblock {\em Rend. Acc. Naz. delle Scienze detta dei XL, Mem. di Mat.},
  108:Vol. XI, fasc. 14, 253--294, 1990.

\bibitem{GAGM}
R.~Giuliano~Antonini, G.~Grekos, and L.~Mi\v{s}\'\i{}k.
\newblock On weighted densities.
\newblock {\em Czechosl. Math. J.}, 57(3):947--962, 2007.

\bibitem{GREDENSSURV}
G.~Grekos.
\newblock On various definitions of density (survey).
\newblock {\em Tatra Mt. Math. Publ.}, 31:17--27, 2005.

\bibitem{HRJECH}
K.~Hrbacek and T.~Jech.
\newblock {\em Introduction to set theory}.
\newblock Marcel Dekker, New York, 1999.

\bibitem{JERISONCONVEX}
M.~Jerison.
\newblock A property of extreme points of compact convex sets.
\newblock {\em Proc. Amer. Math. Soc.}, 5(5):782--783, 1954.

\bibitem{KOMJTOT}
P.~Komj{\'a}th and V.~Totik.
\newblock {\em Problems and Theorems in Classical Set Theory}.
\newblock Springer, 2006.
\newblock Problem Books in Mathematics.

\bibitem{KMSS}
P.~Kostyrko, M.~Ma\v{c}aj, T.~{\v{S}}al\'at, and M.~Sleziak.
\newblock {$\mathcal I$}-convergence and extremal {$\mathcal I$}-limit points.
\newblock {\em Math. Slov.}, 55(4):443--464, 2005.

\bibitem{LAUWERS}
L.~Lauwers.
\newblock Intertemporal objective functions: strong {Pareto} versus anonymity.
\newblock {\em Mathematical Social Sciences}, 35:37--55, 1998.

\bibitem{MAHARAM}
D.~Maharam.
\newblock {Finitely additive measures on the integers}.
\newblock {\em Sankhya, Ser. A}, 38:44--59, 1976.

\bibitem{MORRISON}
T.~J. Morrison.
\newblock {\em Functional Analysis: An Introduction to {B}anach Space Theory}.
\newblock Wiley, 2000.

\bibitem{POLYA}
G.~P{\'o}lya.
\newblock {Untersuchungen {\"u}ber {L}{\"u}cken und {S}ingularit{\"a}ten von
  {P}otenzreihen.}
\newblock {\em Math. Zeit.}, 29:549--640, 1929.

\bibitem{RAJAGOPAL}
C.~T. Rajagopal.
\newblock Some limit theorems.
\newblock {\em Amer. J. Math.}, 70(1):157--167, 1948.

\bibitem{SALTIJD}
T.~{\v{S}}al\'at and R.~Tijdeman.
\newblock Asymptotic densities of sets of positive integers.
\newblock {\em Math. Slov.}, 33:199--207, 1983.

\bibitem{SLEZZIMDENSLEVY}
M.~Sleziak and M.~Ziman.
\newblock {L}\'evy group and density measures.
\newblock {\em J. Number Theory}, 128(12):3005--3012, 2008.

\bibitem{SLEZZIMDENSRANGE}
M.~Sleziak and M.~Ziman.
\newblock Range of density measures.
\newblock {\em Acta Mathematica Universitatis Ostraviensis}, 17:33--50, 2009.

\bibitem{SUREKHABHARAO}
K.~Surekha and K.P.S.~Bhaskara Rao.
\newblock May's theorem in an infinite setting.
\newblock {\em Journal of Mathematical Economics}, 46:50--55, 2010.

\bibitem{VANDOUWEN}
E.~K. van Douwen.
\newblock Finitely additive measures on {$\mathbb{N}$}.
\newblock {\em Topology and its Applications}, 47:223--268, 1992.

\end{thebibliography}

\end{document}